\documentclass[12pt]{amsart}
\usepackage{amssymb}
\usepackage{hyperref}
\usepackage{fullpage, enumerate,xcolor}

\newtheorem{theorem}{Theorem}[section]
\numberwithin{equation}{section}

\newcommand{\bbQ}{\ensuremath{\mathbb{Q}}}
\newcommand{\bbC}{\ensuremath{\mathbb{C}}}
\newcommand{\bbQp}{\ensuremath{{\mathbb{Q}_p}}}
\newcommand{\bbQpbar}{{\ensuremath{\overline{\mathbb{Q}}_p}}}

\newcommand{\bbZ}{\ensuremath{\mathbb{Z}}}
\newcommand{\bbZp}{\ensuremath{{\mathbb{Z}_p}}}
\newcommand{\bbZpbar}{\ensuremath{{\overline{\mathbb{Z}}_p}}}

\newcommand{\Fr}{\ensuremath{\mathrm{Fr}}}

\newcommand{\tr}{\ensuremath{\mathrm{tr}}}

\newcommand{\xra}{\xrightarrow}

\newcommand{\gln}{\ensuremath{\operatorname{GL}}}

\begin{document}
\title{On Frobenius semisimplicity in Hida families}
\author{Jyoti Prakash Saha}
\address{Department of Mathematics, Ben-Gurion University of the Negev, Be'er Sheva 8410501, Israel}
\curraddr{}
\email{jyotipra@post.bgu.ac.il}
\thanks{}

\subjclass[2010]{11F11, 11F80}

\keywords{Frobenius action, Ordinary cusp forms, Hida families}

\begin{abstract}
Let $p\geq 5$ be a prime and $\ell\neq p$ be a prime not dividing the tame level of a $p$-ordinary Hida family. We prove that the actions of the Frobenius element at $\ell$ on the Galois representations attached to almost all arithmetic specializations are semisimple and non-scalar. If $k_f$ denotes the weight of a cusp form $f(z)= \sum_{n\geq 1} a_\ell(f) e^{2\pi i n z}$, then the inequality 
$$|a_\ell(f) | \leq 2 \ell^{(k_f-1)/2},$$
predicted by the Ramanujan conjecture, is a strict inequality for almost all members $f$ of the family. 
\end{abstract}

\maketitle

\section{Introduction}
The study of the Galois representations of the absolute Galois groups of number fields is of central importance in number theory. The \'etale cohomologies of algebraic varieties provide examples of compatible systems of Galois representations. According to the Fontaine--Mazur conjecture \cite[Conjecture 1]{FontaineMazurGeometricGaloisReprs}, any irreducible $p$-adic Galois representation of the absolute Galois group of a number field $F$ which is geometric (in the sense of \cite[p. 193]{FontaineMazurGeometricGaloisReprs}) is isomorphic to a subquotient of an \'etale cohomology group of a smooth projective variety over $F$ with coefficients in $\bbQp(r)$ for some $r\in \bbZ$. 

If $f$ is a normalized eigen new cusp form of weight $\geq 2$ and $\mathfrak{p}$ is a prime of its number field $K$, then by the works of Eichler \cite{Eichler54}, Shimura \cite{ShimuraCorrespondances} and Deligne \cite{DeligneModFormAndlAdicRepr}, there exists a continuous Galois representation $\rho_{f, \mathfrak{p}}: \mathrm{Gal}(\overline{\mathbb{Q}}/\bbQ) \to \gln_2(K_\mathfrak{p})$ which is absolutely irreducible by a result of Ribet \cite[Theorem 2.3]{RibetGalReprAttachedToEigenformsWithNebentypus} (we refer to \S \ref{Sec: Modular forms} for further properties of $\rho_{f, \mathfrak{p}}$). In \cite{SchollMotivesForModForm}, Scholl constructed motives attached to normalized eigen new cusp forms of weight $\geq 2$. It is expected that the unramified Frobenius elements act semisimply on the \'etale realizations of motives (see \cite[Question 12.4]{SerreMotivesVol} of Serre, the conjecture of Grothendieck--Serre \cite[Conjecture 12.5]{JannsenMixedMotivesAlgKTheory}). If $f$ is of weight two, or $f$ has weight $\geq 3$ and Tate's conjecture \cite[Conjecture 1.14]{MilneMotivesOverFiniteFields} holds, then Coleman and Edixhoven \cite[Theorem 4.1]{ColemanEdixhoven} proved that the actions of the Frobenius elements at the unramified places of $\rho_{f, \mathfrak{p}}$ are semisimple and non-scalar. When $f$ is of level one and has weight $\geq 2$, Gouv\^ea proved that the actions of the unramified Frobenius elements on $\rho_{f, \mathfrak{p}}$ are semisimple and non-scalar \cite[Theorem 1]{GouveaWhereTheSlopesAre}. 

If $f$ has weight $k\geq 2$ and $\ell$ is a prime not dividing the level of $f$, then the Ramanujan conjecture predicts that the $\ell$-th Fourier coefficient $a_\ell(f)$ of $f$ satisfies 
\begin{equation}
\label{Pure}
|a_\ell(f)| \leq 2\ell^{(k-1)/2}.
\end{equation} 
Deligne proved this conjecture by showing that the roots of the characteristic polynomial of the action of Frobenius element $\Fr_\ell$ on $\rho_{f, \mathfrak p}$ have absolute values equal to $\ell^{(k-1)/2}$ (for any $\mathfrak{p}$ not containing $\ell$) \cite[Theorem 1.6]{DeligneWeil1}. Note that the characteristic polynomial of $\rho_{f, \mathfrak p}(\Fr_\ell)$ has simple roots if and only if the inequality in Equation \eqref{Pure} is strict. In other words, the action of the Frobenius element $\Fr_\ell$ on $\rho_{f, \mathfrak{p}}$ is semisimple and non-scalar if and only if the inequality in Equation \eqref{Pure} is strict. Hence from the results of \cite{ColemanEdixhoven, GouveaWhereTheSlopesAre}, it follows that the inequality in Equation \eqref{Pure} is strict if $f$ is of level one, or if $f$ has weight two, or if $f$ has weight $\geq 3$ and Tate's conjecture holds. 

Let $\ell\neq p$ be a prime not dividing the tame level of a $p$-ordinary Hida family with $p\geq 5$. We prove in Theorem \ref{Thm: Hida} that the Frobenius element at $\ell$ acts with distinct eigenvalues on the Galois representations associated with almost all arithmetic specializations of the family, and consequently, the inequality in the Ramanujan conjecture at the prime $\ell$ is strict for almost all members of the Hida family. 

\section{Modular forms}
\label{Sec: Modular forms}
A holomorphic function $f: \mathfrak{h} \to \bbC$ defined on the upper-half plane $\mathfrak{h}$ is called a \textit{cusp form} of weight $k\geq 1$ and level $N\geq 1$ if 
$$f\left(\dfrac{az+b}{cz+d}\right) = (cz+d)^k f(z), \quad z\in \mathfrak{h}$$
holds for all elements 
$
\left(
\begin{smallmatrix}
a  & b \\
c & d 
\end{smallmatrix}
\right)\in \Gamma_1(N):= \ker \left(\mathrm{SL}_2(\bbZ) \xra{\mathrm{mod}\, N} \mathrm{SL}_2(\bbZ/N\bbZ)\right)$, and $f$ is holomorphic at the cusps of $\Gamma_1(N)$ and vanishes at $\infty$. The space $S_k(\Gamma_1 (N))$ of cusp forms of weight $k$ and level $N$ is a finite dimensional complex vector space. This space is equipped with the action of the Hecke operators $T_n$ for each integer $n \geq 1$. An element $f\in S_k(\Gamma_1(N))$ is called an \textit{eigenform} if it is an eigenvector for $T_n$ for any $n\geq 1$. If $f(z) = \sum_{n\geq 1} a_n e^{2\pi i nz}$ is a normalized eigen new form, then its Fourier coefficients generate a finite extension of $\bbQ$, called the \textit{number field} of $f$. For more details, we refer to \cite{Shimurabookmodform}. 

We choose algebraic closures $\overline{\mathbb{Q}}$ of $\bbQ$ and $\bbQpbar$ of $\mathbb{Q}_p$. Fix embeddings $i_\infty: \overline{\mathbb{Q}} \hookrightarrow \bbC, i_p: \overline{\mathbb{Q}} \hookrightarrow \bbQpbar$. A normalized eigenform in $S_k(\Gamma_1(N))$ is called \textit{$p$-ordinary} if its $p$-th Fourier coefficient is a unit in $\bbZpbar$ under the embeddings $i_\infty, i_p$. It is known that if an eigenform is attached to a non-CM rational elliptic curve, then it is ordinary for a density one set of primes \cite[IV-13, Exercise]{SerreAbelianEllAdic}. More generally, if $f$ is attached to a rational elliptic curve, then $f$ is ordinary for infinitely many primes \cite[Exercise 5.11]{SilvermanArithmeticOfEllipticCurves2009}. 

Suppose $f(z) = \sum_{n\geq 1} a_n e^{2\pi i nz}$ is a normalized eigen new form of level $N$ and weight $k\geq 2$ with number field $K$. By the works of Eichler \cite{Eichler54}, Shimura \cite{ShimuraCorrespondances} and Deligne \cite{DeligneModFormAndlAdicRepr}, for each non-archimedean prime $\mathfrak{p}$ of $K$, there exists a continuous Galois representation $\rho_{f, \mathfrak{p}} : \mathrm{Gal}(\overline{\mathbb{Q}}/\bbQ) \to \gln_2(K_\mathfrak{p})$ such that for any prime $\ell$ not dividing the product of $N$ and the prime lying below $\mathfrak{p}$, the representation $\rho_{f, \mathfrak{p}}$ is unramified at $\ell$ and the trace of the action of the Frobenius element $\Fr_\ell$ on $\rho_{f, \mathfrak{p}}$ is equal to $a_\ell$. Ribet proved that this representation is absolutely irreducible \cite[Theorem 2.3]{RibetGalReprAttachedToEigenformsWithNebentypus}. We denote by $\rho_f$ the Galois representation corresponding to the non-archimedean prime of $K$ induced by the embeddings $i_\infty, i_p$. 

\section{Frobenius action in Hida families}
\label{Sec: Hida families}
During 1980's, Hida proved that the ordinary modular forms can be put in $p$-adic families. We review the notion of Hida families following \cite{HidaICM86} and refer to \cite{HidaGalrepreord, HidaIwasawa} for the details. 

Let $p\geq 5$ be a prime and $N$ be a positive integer not divisible by $p$. Let $h^{\mathrm{ord}}(N;\bbZp)$ denote the universal $p$-ordinary Hecke algebra of tame level $N$, which is an algebra over $\bbZp[[X]]$. We denote $h^{\mathrm{ord}}(N;\bbZp)$ by $h^{\mathrm{ord}}$. If $A$ is a $\bbZp[[X]]$-algebra, then a $\bbZp$-algebra homomorphism from $A\to \bbQpbar$ is called an \textit{arithmetic specialization of weight $k$} if it vanishes at $(1+X)^{p^{r-1}} - (1+p)^{(k-2)p^{r-1}}$ for some integers $k\geq 2, r\geq 1$. By \cite[Theorem 2.2]{HidaICM86}, the arithmetic specializations of $h^{\mathrm{ord}}$ are in one-to-one correspondence with the $p$-ordinary $p$-stabilized normalized eigen cusp form of tame level a divisor of $N$ and weight $\geq 2$. Moreover, under this correspondence, the arithmetic specializations of weight $k$ correspond to the ordinary forms of weight $k$. If $\mathfrak{a}$ is a minimal prime ideal of $h^{\mathrm{ord}}$, denote the quotient ring $h^{\mathrm{ord}}/\mathfrak{a}$ by $R(\mathfrak{a})$, and the fraction field of $R(\mathfrak{a})$ by $Q(\mathfrak{a})$. For each minimal prime $\mathfrak{a}$ of $h^{\mathrm{ord}}$, by \cite[Theorem 3.1]{HidaICM86}, there exists a unique (up to equivalence) continuous representation $\rho_\mathfrak{a}: \mathrm{Gal}(\overline{\mathbb{Q}}/\bbQ) \to \gln_2(Q(\mathfrak{a}))$ such that its trace is contained in $R(\mathfrak{a})$, and for any prime $\ell\nmid Np$, the representation $\rho_\mathfrak{a}$ is unramified at $\ell$ and the trace of $\rho_\mathfrak{a}(\Fr_\ell)$ is equal to $T_\ell\, \mathrm{mod}\, \mathfrak{a}$ where $T_\ell\in h^{\mathrm{ord}}$ denotes the Hecke operator at $\ell$. Moreover, for any arithmetic specialization $\lambda: h^{\mathrm{ord}} \to \bbQpbar$ with $\lambda(\mathfrak{a})=0$ for some minimal prime $\mathfrak{a}$ of $h^{\mathrm{ord}}$, the pseudo-representation $\lambda \circ \tr \rho_\mathfrak{a}$ is equal to the trace of the Galois representation $\rho_{f_\lambda}$ where $f_\lambda$ denotes the ordinary form corresponding to $\lambda$. For an arithmetic specialization $\lambda$ of $h^{\mathrm{ord}}$, we denote the weight of the associated ordinary form $f_\lambda$ by $k_\lambda$ and the $q$-expansion of $f_\lambda$ by $f_\lambda(z) = \sum_{n\geq 1} a_\ell(f_\lambda) q^n$ where $q = e^{2\pi i nz}$. 

\begin{theorem}
\label{Thm: Hida}
Let $\ell$ be a prime not dividing $Np$. Then for any minimal prime $\mathfrak{a}$ of $h^{\mathrm{ord}}$, the action of the Frobenius element $\Fr_\ell$ on $\rho_\mathfrak{a}$ has distinct eigenvalues. Consequently, for almost all arithmetic specialization $\lambda$ of $h^{\mathrm{ord}}$, the Frobenius element $\Fr_\ell$ acts on $\rho_{f_\lambda}$ with distinct eigenvalues and the inequality in the Ramanujan conjecture at the prime $\ell$ is strict for $f_\lambda$, i.e., the inequality
\begin{equation}
\label{Pure family}
|a_\ell(f_\lambda) | < 2 \ell ^{(k_\lambda-1)/2}
\end{equation}
holds.
\end{theorem}

\begin{proof}
Denote by $x$ the element $ 2\tr \rho_\mathfrak{a} (\Fr_\ell^2)- (\tr \rho_\mathfrak{a} (\Fr_\ell))^2$ of $R(\mathfrak{a})$. Note that $x$ is equal to the square of the difference of the eigenvalues of $\rho_\mathfrak{a}(\Fr_\ell)$ and for any arithmetic specialization $\lambda$ of $R(\mathfrak{a})$, $\lambda(x)$ is equal to the square of the difference of the eigenvalues of $\rho_{f_\lambda}(\Fr_\ell)$ since $\lambda\circ \tr \rho_\mathfrak{a} = \tr \rho_{f_\lambda}$. Since $h^{\mathrm{ord}}$ is of finite type over $\bbZp[[X]]$ (by \cite[Theorem 2.2]{HidaICM86}), its quotient $R(\mathfrak{a})= h^{\mathrm{ord}}/\mathfrak{a}$ is also of finite type over $\bbZp[[X]]$. So the induced map $\mathrm{Spec} R(\mathfrak{a}) \to \mathrm{Spec} \bbZp[[X]]$ is surjective and has finite fibres. Thus $R(\mathfrak{a})$ has arithmetic specializations of any weight $\geq 2$. Let $\lambda'$ be an arithmetic specialization of $R(\mathfrak{a})$ of weight two. So the associated ordinary form $f_{\lambda'}$ is of weight two. Hence, by \cite[Theorem 4.1]{ColemanEdixhoven}, the eigenvalues of $\Fr_\ell$ on $\rho_{f_{\lambda'}}$ are distinct. Then it follows that the image of $x$ under $\lambda'$ is nonzero. So $x$ is a nonzero element of $R(\mathfrak{a})$. This shows that the action of $\Fr_\ell$ on $\rho_\mathfrak{a}$ has distinct eigenvalues for any minimal prime $\mathfrak{a}$ of $h^{\mathrm{ord}}$. 

Note that the zero locus of $x$ defines a finite subset of $\mathrm{Spec} R(\mathfrak{a})$ and a prime ideal of $R(\mathfrak{a})$ could be the kernel of at most a finite number of arithmetic specializations of $R(\mathfrak{a})$. Hence, for almost all arithmetic specialization $\lambda$ of $h^{\mathrm{ord}}$, the image of $x$ under $\lambda$ is nonzero, or equivalently, $\Fr_\ell$ acts on $\rho_{f_\lambda}$ with distinct eigenvalues. 

By \cite[Theorem 1.6]{DeligneWeil1}, the roots of the characteristic polynomial of $\rho_{f_\lambda}(\Fr_\ell)$ have absolute values equal to $\ell^{(k_\lambda-1)/2}$ for any arithmetic specialization $\lambda$. Since these roots are distinct for almost all $\lambda$, it follows that $a_\ell(f_\lambda)$ satisfies the inequality in Equation \eqref{Pure family}. This completes the proof.  
\end{proof}

\section*{Acknowledgement}
The author would like to acknowledge the support provided by a postdoctoral fellowship at the Ben-Gurion University of the Negev, offered by the Israel Science Foundation Grant number 87590011 of Prof.\,Ishai Dan-Cohen. 

\def\cprime{$'$} \def\Dbar{\leavevmode\lower.6ex\hbox to 0pt{\hskip-.23ex
  \accent"16\hss}D} \def\cfac#1{\ifmmode\setbox7\hbox{$\accent"5E#1$}\else
  \setbox7\hbox{\accent"5E#1}\penalty 10000\relax\fi\raise 1\ht7
  \hbox{\lower1.15ex\hbox to 1\wd7{\hss\accent"13\hss}}\penalty 10000
  \hskip-1\wd7\penalty 10000\box7}
  \def\cftil#1{\ifmmode\setbox7\hbox{$\accent"5E#1$}\else
  \setbox7\hbox{\accent"5E#1}\penalty 10000\relax\fi\raise 1\ht7
  \hbox{\lower1.15ex\hbox to 1\wd7{\hss\accent"7E\hss}}\penalty 10000
  \hskip-1\wd7\penalty 10000\box7}
  \def\polhk#1{\setbox0=\hbox{#1}{\ooalign{\hidewidth
  \lower1.5ex\hbox{`}\hidewidth\crcr\unhbox0}}}

\end{document}